\documentclass[11pt]{article}
\usepackage{geometry}                % See geometry.pdf to learn the layout options. There are lots.
\usepackage{graphicx}
\usepackage{enumerate}
\usepackage{setspace}
\usepackage{amssymb}
\usepackage{amsmath, amsthm}
\usepackage{amsmath}
\usepackage{amsthm}
\usepackage{blkarray}
\usepackage{epstopdf}
\usepackage{footnote}
\newtheorem{thm}{Theorem}

\newtheorem*{thm1}{Theorem A}

\newtheorem{definition}{Definition}

\newtheorem{lemma}{Lemma}

\begin{document}\large{
%%%%%%%%%%%%%%%%%%                            Title                                  %%%%%%%%%%%%%%%%%%%%%%%%%%%%%
\title{Convergence of row sequences of simultaneous Pad\'{e}-Faber approximants \footnote{This paper is accepted and will be published in Journal 
``Mathematical Notes".}}
\author{N. Bosuwan\thanks{The research of N. Bosuwan was supported by the Strengthen Research Grant for New Lecturer from the Thailand Research Fund and the Office of the Higher Education Commission (MRG6080133) and Faculty of Science, Mahidol University.}\,\,\,\footnote{Corresponding author.}}   }
\maketitle

\centerline {Department of Mathematics, Faculty of Science, Mahidol University}
\centerline {Rama VI Road, Ratchathewi District, Bangkok 10400, Thailand}
\centerline {e-mail : {\tt nattapong.bos@mahidol.ac.th}}
\centerline {Centre of Excellence in Mathematics, CHE}
\centerline {Si Ayutthaya Road, Bangkok 10400, Thailand}

%%%%%%%%%%%%%%%%%%%%%%%%       Abstract                            %%%%%%%%%%%%%%%%%%%%%%%%%%%%%%

\section*{Abstract}

We consider row sequences of vector valued Pad\'{e}-Faber approximants (simultaneous Pad\'{e}-Faber approximants) and prove a Montessus de Ballore type theorem.

\quad

 \section*{Keywords}{Montessus de Ballore's theorem; Pad\'{e}-Faber approximants; Simultaneous Pad\'{e} approximants; Hermite-Pad\'{e} approximants.}

 \section*{Mathematics Subject Classification:}
30E10; 41A21.

%%%%%%%%%%%%%%%%%                                    Introduction           %%%%%%%%%%%%%%%%%%%%%%%%%%%%%%
\section{Introduction}

\quad Let $E$ be a compact set of the complex plane $\mathbb{C}$ such that $\overline{\mathbb{C}}\setminus E$ is simply connected and $E$ contains more than one point. There exists a unique exterior conformal mapping $\Phi$ from $\overline{\mathbb{C}}\setminus E$ onto $\overline{\mathbb{C}}\setminus \{w\in \mathbb{C}: |w|\leq 1\}$ satisfying $\Phi(\infty)=\infty$ and $\Phi'(\infty)>0.$ For any $\rho>1,$  we denote by 
 $$\Gamma_{\rho}:=\{z\in \mathbb{C}: |\Phi(z)|=\rho\}, \quad \quad \mbox{and} \quad \quad D_{\rho}:=E\cup \{z\in \mathbb{C}: |\Phi(z)|<\rho\},$$
 a \emph{level curve with respect to $E$ of index $\rho$} and a \emph{canonical domain with respect to $E$ of index $\rho$}, respectively. The Faber polynomials (see \cite{Suetin}) for $E$ are defined by the formulas 
 \begin{equation}\label{polynomial}
 \Phi_n(z):=\frac{1}{2\pi i}\int_{\Gamma_{\rho}} \frac{\Phi^n(t)}{t-z}dt, \quad\quad z\in D_{\rho}, \quad \quad n=0,1,2,\ldots. 
 \end{equation} Denote by $\mathcal{H}(E)$ the space of all functions holomorphic in some neighborhood of $E.$ We define
$$\mathcal{H}(E)^d:=\{(F_1,F_2,\ldots,F_d): \textup{$F_\alpha\in \mathcal{H}(E)$ for all $\alpha=1,2,\ldots,d$}\}$$
and the set of all nonnegative integers is denoted by $\mathbb{N}.$ 

\begin{definition}\label{simuf}\textup{ Let $\textup{\textbf{F}}=(F_1,F_2,\ldots,F_d)\in \mathcal{H}(E)^d.$ Fix a multi-index $\textup{\textbf{m}}=(m_1,m_2,\ldots, $ $m_d)\in \mathbb{N}^d \setminus \{\textup{\textbf{0}}\},$ where  $\textup{\textbf{0}}$ is the zero vector in $\mathbb{N}^d.$ Set $|\textup{\textbf{m}}|=m_1+m_2+\ldots+m_d.$ Then, for each $n\geq \max\{m_1,m_2,\ldots,m_d\},$ there exist polynomials $Q_{n,\textup{\textbf{m}}}$ and $ P_{n,\textup{\textbf{m}},\alpha},$ $\alpha=1,2,\ldots,d$ such that 
$$\deg(P_{n,\textup{\textbf{m}},\alpha})\leq n-m_\alpha, \quad \deg(Q_{n,\textup{\textbf{m}}})\leq |\textup{\textbf{m}}|, \quad Q_{n,\textup{\textbf{m}}}\not\equiv 0,$$
$$ Q_{n,\textup{\textbf{m}}} F_\alpha-P_{n,\textup{\textbf{m}},\alpha}=a_{n+1,n}^{(\alpha)}\Phi_{n+1}(z)+a_{n+2,n}^{(\alpha)}\Phi_{n+2}(z)+\ldots,$$
for all $\alpha=1,2,\ldots,d.$
The vector of rational functions 
$$\textup{\textbf{R}}_{n,\textup{\textbf{m}}}:=(R_{n,\textup{\textbf{m}},1},R_{n,\textup{\textbf{m}},2},\ldots,R_{n,\textup{\textbf{m}},d})=(P_{n,\textup{\textbf{m}},1}/Q_{n,\textup{\textbf{m}}},P_{n,\textup{\textbf{m}},2}/Q_{n,\textup{\textbf{m}}},\ldots,P_{n,\textup{\textbf{m}},d}/Q_{n,\textup{\textbf{m}}})$$ is called an \emph{$(n, \textup{\textbf{m}})$ (linear) simultaneous Pad\'{e}-Faber approximant of $\textup{\textbf{F}}.$}}
\end{definition}

In fact, the numbers $a_{k,n}^{(\alpha)}$ depend on $\textup{\textbf{m}}$ but to simplify the notation we will not indicate it. It is easy to see that if $d=1,$ then the linear simultaneous Pad\'{e}-Faber approximants reduce to the linear Pad\'{e}-Faber approximants with a slight modification on the index $n$ (see, e.g., \cite{Sutinpade} for the definition of linear Pad\'{e}-Faber approximants). Moreover, for the case when $d=1,$ we would like to point out that there is another related construction called nonlinear Pad\'{e}-Faber approximants (see \cite{nonlinear}). Unlike the classical case, these linear and nonlinear Pad\'{e}-Faber approximants lead, in general, to different rational functions (see examples in \cite{nonlinear} and \cite{Suetin2009}). Because we will restrict our attention in this paper to linear simultaneous Pad\'{e}-Faber approximants,  in the sequel, we will omit the word ``linear'' when we refer to them.

For any pair $(n,\textup{\textbf{m}}),$ a vector of rational functions $\textup{\textbf{R}}_{n,\textup{\textbf{m}}}$ always exists but, in general, it may not be unique. In what follows, we assume that given $(n,\textup{\textbf{m}}),$ one solution is taken. 

%For each $i=1,2,\ldots,d,$ $P_{n,\textup{\textbf{m}},i}$ is uniquely determined by $Q_{n,\textup{\textbf{m}}}.$ 

%Basically, the concept of simultaneous Pad\'{e}-Faber approximation follows from the combination of simultaneous Pad\'{e} approximation in the context of Taylor expansions (see, e.g., \cite{MorrisSaff}) and Pad\'{e}-Faber approximation (see, e.g., \cite{Sutinpade}). In \cite{Sutinpade}, Suetin proved a Montessus 
Now, let us introduce a definition of a pole and its order for a vector of functions.
\begin{definition}\textup{ Let $\textup{\textbf{F}}=(F_1,F_2,\ldots,F_d)\in \mathcal{H}(E)^d$ be a vector of functions meromorphic in some domain $D.$  We say that \emph{$\lambda$ is a pole of $\textup{\textbf{F}}$ in $D$ of order $\tau$} if there exists an index $\alpha\in \{1,2,\ldots,d\}$ such that $\lambda$ is a pole of $F_\alpha$ in $D$ of order $\tau$  and for the rest of the  indices $j\not=\alpha$, either $\lambda$ is not a pole of $F_j$ or  $\lambda$ is a pole of $F_j$ with order less than or equal to $\tau.$}
\end{definition}

Let $\textup{\textbf{F}}\in \mathcal{H}(E)^d.$ Denote by $\rho_{|\textup{\textbf{m}}|}(\textup{\textbf{F}})$ the index $\rho>1$ of the largest canonical domain $D_{\rho}$ inside of which $\textup{\textbf{F}}$ has at most $|\textup{\textbf{m}}|$ poles. Let $\lambda_1,\lambda_2,\ldots, \lambda_q$ be the distinct poles of $\textup{\textbf{F}}$ in $D_{\rho_{|\textup{\textbf{m}}|}(\textup{\textbf{F}})}$ and let $$L:=\left(1+\min_{j=1,2,\ldots,q} |\Phi(\lambda_j)|\right)/2.$$ The set of these poles is denoted by $\mathcal{P}_{|\textup{\textbf{m}}|}(\textup{\textbf{F}}).$ The normalization of $Q_{n,\textup{\textbf{m}}}$ used in this paper in terms of its zeros $\lambda_{n,j}$ is the following:
\begin{equation}\label{nor}
Q_{n,\textup{\textbf{m}}}(z):=\prod_{|\Phi(\lambda_{n,j})|\leq L} (z-\lambda_{n,j}) \prod_{|\Phi(\lambda_{n,j})|> L} \left(1-\frac{z}{\lambda_{n,j}}\right).
\end{equation}
Denote by $Q_{|\textup{\textbf{m}}|}^{\textup{\textbf{F}}}$  the polynomial whose zeros are the poles of $\textup{\textbf{F}}$ in $D_{\rho_{|\textup{\textbf{m}}|}(\textup{\textbf{F}})}$ counting multiplicities normalized as in \eqref{nor}. 

Before going into details, let us describe the convergence of row sequences of Pad\'{e}-Faber approximants which is corresponding to the simultaneous Pad\'{e}-Faber approximants for the scalar case ($d=1$). When $d=1,$ we write $\textup{\textbf{F}}=F,$ $|\textup{\textbf{m}}|=\textup{\textbf{m}}=m\in \mathbb{N},$ $\mathcal{P}_{|\textup{\textbf{m}}|}(\textup{\textbf{F}})=\mathcal{P}_{m}(F),$ $\rho_{|\textup{\textbf{m}}|}(\textup{\textbf{F}})=\rho_{m}(F),$  and  $\textup{\textbf{R}}_{n,\textup{\textbf{m}}}=R_{n,m}.$ An analogue of Montessus de Ballore's theorem for Pad\'{e}-Faber approximants proved by Suetin \cite{Sutinpade} is the following:

\begin{thm1}\label{montessusana}
 Suppose $F\in \mathcal{H}(E)$ has poles of total multiplicity exactly $m$ in $D_{\rho_{m}(F)}.$ Then, $R_{n,m}$ is uniquely determined for all sufficiently large $n$ and the sequence $R_{n,m}$ converges uniformly to $F$ inside $D_{\rho_m(F)}\setminus \mathcal{P}_m(F)$ as $n \rightarrow \infty.$ Moreover, for any compact subset $K$ of $D_{\rho_m(F)}\setminus \mathcal{P}_m(F),$
\begin{equation}\label{asymnroot}
\limsup_{n\rightarrow \infty} \|F-R_{n,m}\|^{1/n}_{K}\leq \frac{\|\Phi\|_K}{\rho_m(F)},
\end{equation}
where $\|\cdot \|_{K}$ denotes the sup-norm on $K$ and if $K\subset E,$ then $\|\Phi\|_K$ is replaced by $1.$ 
\end{thm1}

Here and in what follows, the phrase ``uniformly inside a domain" means ``uniformly on each compact subset of the domain". 
The goal of this paper is to extend the above result from the scalar case to the vector case.

In \cite{MorrisSaff}, Graves-Moris and Saff proved a Montessus de Ballore type theorem for simultaneous Pad\'{e} approximants (in the context of Taylor expansions) using the concept of polewise independent of a vector of functions. We adapt their notion to fit our type of regions. 

\begin{definition}\textup{ Let $\textup{\textbf{F}}=(F_1,F_2,\ldots,F_d)\in\mathcal{H}(E)^d$ be a vector of functions meromorphic in some canonical domain $D_{\rho}$ and let $\textup{\textbf{m}}=(m_1,m_2,\ldots,m_d)\in \mathbb{N}^d \setminus \{\textup{\textbf{0}}\}$ be the multi-index.  Then the function $\textup{\textbf{F}}$ is said to be \emph{polewise independent with respect to the multi-index $\textup{\textbf{m}}$ in $D_\rho$} if and only if there do not exist  polynomials $v_1,v_2,\ldots,v_d$ at least one of which is non-null, satisfying 
\begin{enumerate}
\item [(i)] $\deg v_\alpha \leq m_\alpha-1,$ $\alpha=1,2,\ldots,d,$ if $m_\alpha\geq 1,$
\item [(ii)] $v_\alpha \equiv 0$ if $m_\alpha=0,$ 
\item [(iii)] $\sum_{\alpha=1}^{d} (v_\alpha\circ \Phi) \cdot  F_\alpha  \in \mathcal{H}(D_{\rho}\setminus E),$
\end{enumerate}
where $\mathcal{H}(D_{\rho}\setminus E)$ is the space of all holomorphic functions in $D_{\rho}\setminus E.$
}
\end{definition}

Our main result served as the extension of Theorem A  is the following:

\begin{thm}\label{thm1.4} Let $\textup{\textbf{F}}=(F_1,F_2,\ldots,F_d)\in \mathcal{H}(E)^d$ be a vector of functions meromorphic in $D_{\rho_{|\textup{\textbf{m}}|}(\textup{\textbf{F}})}$ and $\textup{\textbf{m}}\in \mathbb{N}^d\setminus \{\textup{\textbf{0}}\}$ be a fixed multi-index. Suppose that $\textup{\textbf{F}}$ is polewise independent with respect to the multi-index $\textup{\textbf{m}}$ in $D_{\rho_{|\textup{\textbf{m}}|}(\textup{\textbf{F}})}.$ Then, $\textup{\textbf{R}}_{n,\textup{\textbf{m}}}$ is uniquely determined for all sufficiently large $n$ and for each $\alpha=1,2,\ldots,d,$  $R_{n,\textup{\textbf{m}},\alpha}$ converges uniformly to $F_\alpha$ inside $D_{\rho_{|\textup{\textbf{m}}|}(\textup{\textbf{F}})}\setminus \mathcal{P}_{|\textup{\textbf{m}}|}(\textup{\textbf{F}}).$ Moreover, for each $\alpha=1,2,\ldots,d$ and for any compact set $K\subset D_{\rho_{|\textup{\textbf{m}}|}(\textup{\textbf{F}})}\setminus \mathcal{P}_{|\textup{\textbf{m}}|}(\textup{\textbf{F}}),$
\begin{equation}\label{2.4}
\limsup_{n \rightarrow \infty} \|F_\alpha-R_{n,\textup{\textbf{m}},\alpha}\|_{K}^{1/n}\leq \frac{\|\Phi\|_K}{\rho_{|\textup{\textbf{m}}|}(\textup{\textbf{F}})},
\end{equation}
where $\|\cdot\|_K$ denotes the sup-norm on $K$ and if $K\subset E,$ then $\|\Phi\|_K$ is replaced by $1.$ Additionally,
\begin{equation}\label{2.5}
\limsup_{n \rightarrow \infty} \|Q_{n,\textup{\textbf{m}}}-Q_{|\textup{\textbf{m}}|}^{\textup{\textbf{F}}}\|^{1/n}\leq \frac{
\max_{\lambda\in \mathcal{P}_{|\textup{\textbf{m}}|}(\textup{\textbf{F}})} |\Phi(\lambda)|}{\rho_{|\textup{\textbf{m}}|}(\textup{\textbf{F}})},
\end{equation}
where $\|\cdot\|$ denotes (for example) the  norm induced in the space of polynomials of degree at most $|\textup{\textbf{m}}|$ by the maximum of the absolute value of the coefficients.
\end{thm}

Since the space of polynomials of degree at most $|\textup{\textbf{m}}|$ has a finite dimension, all of its norms are equivalent so we can put any norm in \eqref{2.5}.

An outline of this paper is as follows. In the section 2, we introduce some more notation and auxiliary lemmas. The proof of the main result is in the section 3.

%%%%%%%%%%%%%%%%%%%%%%%%%%%%%%%%%%%%%%%%%%%%%%%%%%%%

\section{Notation and auxiliary results}\label{123}

\quad \quad  First, let us discuss some properties of Faber polynomial expansions of holomorphic functions which play a major role in our proof. The \emph{Faber coefficient} of $G\in \mathcal{H}(E)$ with respect to $\Phi_n$ is given by
\begin{equation*}\label{Fourierco}
[G]_n:=\frac{1}{2\pi i}\int_{\Gamma_\rho} \frac{G(t) \Phi'(t)}{\Phi^{n+1}(t)} dt,
\end{equation*}
where $\rho\in (1,\rho_{0}(G)).$
The following lemma (see, e.g., \cite{SmirnovLebedev}) is obtained in the same way as similar statements are proved for Taylor series.
\begin{lemma}\label{expan} Let $G\in \mathcal{H}(E).$ Then,
$$\rho_0(G)=\left(\limsup_{n \rightarrow \infty} |[G]_n|^{1/n} \right)^{-1}.$$
Moroever, the series $\sum_{n=0}^{\infty} [G]_n \Phi_n$ converges to $G$ uniformly inside  ${D}_{\rho_{0}(G)}.$
\end{lemma}

As a consequence of Lemma \ref{expan},
if $\textup{\textbf{F}}=(F_1,F_2,\ldots,F_d)\in \mathcal{H}(E)^d$, then for each $\alpha=1,2,\ldots,d,$ 
\begin{equation}\label{usethisasdef}
Q_{n,\textup{\textbf{m}}}(z)F_\alpha(z)-P_{n,\textup{\textbf{m}},\alpha}(z)=\sum_{k=n+1}^{\infty} [ Q_{n,\textup{\textbf{m}}}F_\alpha]_{k}\,\Phi_k(z), \quad \quad z\in D_{\rho_{0}(F_\alpha)},
\end{equation}
and $P_{n,\textup{\textbf{m}},\alpha}=\sum_{k=0}^{n-m_\alpha} [ Q_{n,\textup{\textbf{m}}} F_\alpha]_{k}\,\Phi_k$ is uniquely determined by $Q_{n,\textup{\textbf{m}}}.$

Next, let us introduce a concept of convergence in $h$-content. Let $B$ be a subset of the complex plane $\mathbb{C}$. By $\mathcal{U}(B),$ we denote the class of all
coverings of $B$ by at most a numerable set of disks. Set  
$$h(B):= \inf\left\{\sum_{j=1}^{\infty} |U_j|:\{U_j\}\in \mathcal{U}(B) \right\},$$  
where $|U_j|$ stands for the radius of the disk $U_j.$ The quantity $h(B)$ is called the $1$-dimensional Hausdorff content of the set $B.$ This set function is not a measure but it is semi-additive and monotonic. 

\begin{definition}\textup{ Let $\{g_n\}_{n\in \mathbb{N}}$ be a sequence of complex valued functions defined on a domain $D \subset \mathbb{C}$ and $g$ another complex function defined on $D.$ We say that \emph{$\{g_n\}_{n\in \mathbb{N}}$ converges in $h$-content to the function $g$ on compact subsets of $D$} if for every compact subset $K$ of $D$ and for each $\varepsilon>0,$ we have
$$\lim_{n \rightarrow \infty} h\{z\in K:|g_n(z)-g(z)|>\varepsilon\}=0.$$
Such a convergence will be denoted by $h$-$\lim_{n \rightarrow \infty} g_n = g$ in $D.$}
\end{definition}

The next lemma proved by Gonchar (see Lemma 1 in \cite{Gonchar1} or in Section \S2., subsection 2, part b in \cite{Aag81}) allows us to derive uniform convergence on compact subsets of the region under consideration from convergence in $h$-content under appropriate assumptions. 
\begin{lemma}\label{goncharlemma} Suppose that $h$-$\lim_{n \rightarrow \infty} g_n = g$ in $D.$ 
%Then the following assertions hold true:
%\begin{enumerate}
%\item [(i)] If the functions $g_n, n \in \mathbb{N},$ are holomorphic in $D,$ then the sequence $\{g_n\}$ converges uniformly on compact subsets of $D$ and $g$ is holomorphic in $D$ (more precisely, it is equal to a holomorphic function in $D$ except on a set of $h$-content zero).
%\item [(ii)] 
If each of the functions $g_n$ is meromorphic in $D$ and has no more than $k < +\infty$ poles in this domain, then the limit function $g$ is (except on a set of $h$-content zero) also meromorphic and has no more than $k$ poles in $D.$ Hence, in particular, if $g$ has a pole of order $\nu$ at the point $\lambda\in D,$ then at least $\nu$ poles of $g_n$ tend to $\lambda$ as $n \rightarrow \infty.$
%[(iii)]
%If each function $g_n$ is meromorphic and has no more than $k < +\infty$ poles in $D$ and the function $g$ is meromorphic and has exactly $k$ poles in $D,$ then all $g_n, n \geq N,$ also have $k$ poles in $D;$ the poles of $g_n$ tend to the poles $z_1,z_2,\ldots ,z_k$ of $g$ (taking account of their orders) and the sequence $\{g_n\}$ tends to $g$ uniformly on compact subsets of the domain $D'=D\setminus \{z_1,z_2,\ldots,z_k\}.$
%\end{enumerate}
\end{lemma}

Now, we discuss some upper and lower estimates on the normalized $Q_{n,\textup{\textbf{m}}}$ in \eqref{nor}. We take an arbitrary $\varepsilon>0$ and define an open set $J_{\varepsilon}:=J_{\varepsilon}(\textup{\textbf{F}})$ as follows. For $n\geq |\textup{\textbf{m}}|,$ let $J_{n,\varepsilon}$ denote the $\varepsilon/6|\textup{\textbf{m}}|n^2$-neighborhood of the set of zeros of $Q_{n,\textup{\textbf{m}}}$ and let $J_{|\textup{\textbf{m}}|-1,\varepsilon}$ denote the $\varepsilon/6|\textup{\textbf{m}}|$-neighborhood of the set of poles of $\textup{\textbf{F}}$ in $D_{\rho_{|\textup{\textbf{m}}|}(\textup{\textbf{F}})}.$ Set $J_{\varepsilon}=\cup_{n \geq |\textup{\textbf{m}}|-1} J_{n,\varepsilon}.$ From monotonicity and subadditivity, it is easy to check that $h(J_\varepsilon)<\varepsilon$ and $J_{\varepsilon_1}\subset J_{\varepsilon_2}$ for $\varepsilon_1<\varepsilon_2.$ For any set $B\subset \mathbb{C},$ we put $B(\varepsilon):=B\setminus J_{\varepsilon}.$ Clearly, if $\{g_n\}_{n\in \mathbb{N}}$ converges uniformly to $g$ on $K(\varepsilon)$ for every compact $K\subset D$ and $\varepsilon>0,$ then $h$-$\lim_{n \rightarrow \infty} g_n=g$ in $D.$ 

The normalization of $Q_{n,\textup{\textbf{m}}}$ provides the following useful upper and lower bounds on the estimation of $Q_{n,\textup{\textbf{m}}}.$

\begin{lemma} Let $K\subset \mathbb{C}$ be a  compact set and $\varepsilon>0$ be arbitrary. Then, there exist constants $C_1,C_2>0,$ independent of $n,$ such that 
\begin{equation}\label{fromnormal}
\|Q_{n,\textup{\textbf{m}}}\|_{K}\leq C_1, \quad \quad \min_{z\in K(\varepsilon)}|Q_{n,\textup{\textbf{m}}}(z)|\geq C_2 n^{-2|\textup{\textbf{m}}|},
\end{equation} 
where the second inequality is meaningful when $K(\varepsilon)$ is a non-empty set.
\end{lemma}

\section{Proof of Theorem \ref{thm1.4}}
\begin{proof}[Proof of Theorem \ref{thm1.4}]

From \eqref{usethisasdef}, we have for each $\alpha=1,2,\ldots,d,$ 
\begin{equation}\label{banana1}
Q_{n,\textup{\textbf{m}}}(z)F_\alpha(z)-P_{n,\textup{\textbf{m}},\alpha}(z)=\sum_{k=n+1}^{\infty} a_{k,n}^{(\alpha)}\Phi_k(z), \quad \quad z\in D_{\rho_{0}(F_\alpha)},
\end{equation}
where $$a_{k,n}^{(\alpha)}:=[ Q_{n,\textup{\textbf{m}}}F_\alpha]_{k}=\frac{1}{2\pi i}\int_{\Gamma_\rho} \frac{Q_{n,\textup{\textbf{m}}}(t)F_\alpha(t) \Phi'(t)}{\Phi^{k+1}(t)} dt, \quad \quad \rho\in (1,\rho_0(F_\alpha)).$$
Let $$Q_{|\textup{\textbf{m}}|}^{\textup{\textbf{F}}}(z):=\prod_{j=1}^{q}\left(1-\frac{z}{\lambda_j} \right)^{\tau_j},$$
where $\lambda_1,\lambda_2,\ldots,\lambda_q$ are distinct poles of $\textup{\textbf{F}}$ in $D_{\rho_{|\textup{\textbf{m}}|}(\textup{\textbf{F}})}$ and $\tau_1,\tau_2,\ldots,\tau_q$ are their multiplicities, respectively. Since $\textup{\textbf{F}}$ is polewise independent with respect to $\textup{\textbf{m}}$ in $D_{\rho_{|\textup{\textbf{m}}|}(\textup{\textbf{F}})},$ $\textup{\textbf{F}}$ has exactly $|\textup{\textbf{m}}|$ poles in $D_{\rho_{|\textup{\textbf{m}}|}(\textup{\textbf{F}})}$ and $\sum_{j=1}^{q}\tau_j=|\textup{\textbf{m}}|.$  Multiplying \eqref{banana1} by $Q_{|\textup{\textbf{m}}|}^{\textup{\textbf{F}}}$ and expanding $Q_{|\textup{\textbf{m}}|}^{\textup{\textbf{F}}}Q_{n,\textup{\textbf{m}}}F_\alpha-Q_{|\textup{\textbf{m}}|}^{\textup{\textbf{F}}}P_{n,\textup{\textbf{m}},\alpha}\in \mathcal{H}(D_{\rho_{|\textup{\textbf{m}}|}(\textup{\textbf{F}})})$ in terms of the Faber polynomial system $\{\Phi_{\nu}\}_{\nu=0}^{\infty}$, we obtain that for each $\alpha=1,2,\ldots,d$ and for $z\in D_{\rho_{|\textup{\textbf{m}}|}(\textup{\textbf{F}})},$
$$
Q_{|\textup{\textbf{m}}|}^{\textup{\textbf{F}}}(z)Q_{n,\textup{\textbf{m}}}(z)F_\alpha(z)-Q_{|\textup{\textbf{m}}|}^{\textup{\textbf{F}}}(z)P_{n,\textup{\textbf{m}},\alpha}(z)=\sum_{k=n+1}^{\infty} a_{k,n}^{(\alpha)} Q_{|\textup{\textbf{m}}|}^{\textup{\textbf{F}}}(z)\Phi_k(z)=\sum_{\nu=0}^{\infty} b_{\nu,n}^{(\alpha)} \Phi_\nu(z)
$$
\begin{equation}\label{use}
=\sum_{\nu=0}^{n+|\textup{\textbf{m}}|-m_\alpha} b_{\nu,n}^{(\alpha)} \Phi_\nu(z)+\sum_{\nu=n+|\textup{\textbf{m}}|-m_\alpha+1}^{\infty} b_{\nu,n}^{(\alpha)} \Phi_\nu(z).
\end{equation}
Note that the constants $b_{\nu,n}^{(\alpha)}$ can be calculated in the following forms:
\begin{equation*}\label{firstform}
b_{\nu,n}^{(\alpha)}:=\sum_{k=n+1}^{\infty} a_{k,n}^{(\alpha)} [ Q_{|\textup{\textbf{m}}|}^{\textup{\textbf{F}}} \Phi_k]_{\nu}, \quad \quad \nu=0,1,\ldots,n+|\textup{\textbf{m}}|-m_\alpha,
\end{equation*}
and 
\begin{equation}\label{secondform}
b_{\nu,n}^{(\alpha)}:=      [Q_{|\textup{\textbf{m}}|}^{\textup{\textbf{F}}} Q_{n,\textup{\textbf{m}}} F_\alpha]_{\nu}=\frac{1}{2\pi i}\int_{\Gamma_\rho} \frac{Q_{|\textup{\textbf{m}}|}^{\textup{\textbf{F}}}(t) Q_{n,\textup{\textbf{m}}}(t) F_\alpha(t) \Phi'(t)}{\Phi^{\nu+1}(t)} dt, \quad \quad  \nu \geq n+|\textup{\textbf{m}}|-m_\alpha+1,
\end{equation}
where $\rho\in (1,\rho_{|\textup{\textbf{m}}|}(\textup{\textbf{F}})).$
We want to show that for each $\alpha=1,2,\ldots,d,$ 
\begin{equation}\label{goal}
\limsup_{n \rightarrow \infty} \|\sum_{\nu=0}^{\infty} b_{\nu,n}^{(\alpha)} \Phi_\nu\|_K^{1/n}\leq \frac{\|\Phi\|_K}{\rho_{|\textup{\textbf{m}}|}(\textup{\textbf{F}})},
\end{equation}
for any compact set $K$ such that $E \subset K \subset D_{\rho_{|\textup{\textbf{m}}|}(\textup{\textbf{F}})}.$
Let $K$ be a fixed compact set such that $E \subset K \subset D_{\rho_{|\textup{\textbf{m}}|}(\textup{\textbf{F}})}.$ Let $\rho_1\in (1,\rho_{|\textup{\textbf{m}}|}(\textup{\textbf{F}}))$ be such that 
\begin{equation}\label{rho}
K\cup\{\lambda_1,\lambda_2,\ldots,\lambda_q\}\subset D_{\rho_1}.
\end{equation}
Choose $\delta>0$  sufficiently small so that 
\begin{equation}\label{delta}
 \|\Phi\|_K+\delta<\rho_1-\delta.
\end{equation}
 
We first prove that for each $\alpha=1,2,\ldots,d,$
\begin{equation}\label{simplepart}
\limsup_{n \rightarrow \infty} \|\sum_{\nu=n+|\textup{\textbf{m}}|-m_\alpha+1}^{\infty} b_{\nu,n}^{(\alpha)} \Phi_\nu\|_K^{1/n}\leq \frac{\|\Phi\|_K}{\rho_{|\textup{\textbf{m}}|}(\textup{\textbf{F}})}.
\end{equation}
Using the normalization of $Q_{n,\textup{\textbf{m}}}$ (the first inequality in \eqref{fromnormal}) and  \eqref{secondform} when $\rho=\rho_1$, there exists $n_0\in \mathbb{N}$ such that  for each $\alpha=1,2,\ldots,d,$
\begin{equation}\label{asymps}
|b_{\nu,n}^{(\alpha)}|\leq \frac{c_1}{(\rho_1-\delta)^{\nu}}, \quad \quad \nu\geq n_0
\end{equation}
where $c_1$ does not depend on $n$ (from now on, we will denote some constants that do not depend on $n$ by $c_2,c_3,\ldots$).
Using \eqref{polynomial}, we have
\begin{equation}\label{asymp}
 \|\Phi_{\nu}\|_{K}\leq c_2(\|\Phi\|_K+\delta)^{\nu}, \quad \quad \nu\geq 0.
\end{equation}
Therefore, by \eqref{asymps} and \eqref{asymp}, for $n \geq n_0,$
$$\|\sum_{\nu =n+|\textup{\textbf{m}}|-m_\alpha+1}^{\infty} b_{\nu,n}^{(\alpha)} \Phi_{\nu}\|_K \leq\sum_{\nu = n+|\textup{\textbf{m}}|-m_\alpha+1}^{\infty} |b_{\nu,n}^{(\alpha)}|  \|\Phi_{\nu}\|_K$$
$$\leq \sum_{\nu = n+|\textup{\textbf{m}}|-m_\alpha+1}^{\infty}c_3 \left( \frac{\|\Phi\|_K+\delta}{\rho_1-\delta}\right)^{\nu}\leq c_4  \left( \frac{\|\Phi\|_K+\delta}{\rho_1-\delta}\right)^{n}.$$ Then, for each $\alpha=1,2,\ldots,d,$
$$\limsup_{n \rightarrow \infty}\|\sum_{\nu =n+|\textup{\textbf{m}}|-m_\alpha+1}^{\infty} b_{\nu,n}^{(\alpha)} \Phi_{\nu}\|_K^{1/n}\leq \frac{\|\Phi\|_K +\delta}{\rho_1-\delta}.$$
Letting $\delta\rightarrow 0$ and $\rho_1\rightarrow \rho_{|\textup{\textbf{m}}|}(\textup{\textbf{F}}),$ we have \eqref{simplepart} as we wanted.

Secondly, we show that 
\begin{equation}\label{hardpart}
\limsup_{n \rightarrow \infty} \|\sum_{\nu=0}^{n+|\textup{\textbf{m}}|-m_\alpha} b_{\nu,n}^{(\alpha)} \Phi_\nu\|_K^{1/n}\leq \frac{\|\Phi\|_K}{\rho_{|\textup{\textbf{m}}|}(\textup{\textbf{F}})}. 
\end{equation}
 Recall that $b_{\nu,n}^{(\alpha)}=\sum_{k=n+1}^{\infty} a_{k,n}^{(\alpha)} [ Q_{|\textup{\textbf{m}}|}^{\textup{\textbf{F}}} \Phi_k]_{\nu}.$  Therefore, to approximate $b_{\nu,n}^{(\alpha)},$  we need to approximate $a_{k,n}^{(\alpha)}$ first. We will adapt the technique used in \cite{Sutinpade} to approximate $a_{k,n}^{(\alpha)}.$ Let $\rho_1\in (1,\rho_{|\textup{\textbf{m}}|}(\textup{\textbf{F}}))$ satisfying \eqref{rho} as before. Choose $\rho_2\in (1,\rho_0(\textup{\textbf{F}})).$  We have
$$a_{k,n}^{(\alpha)}=[Q_{n,\textup{\textbf{m}}} F_\alpha]_{k}=\frac{1}{2\pi i}\int_{\Gamma_{\rho_2}} \frac{Q_{n,\textup{\textbf{m}}}(t) F_\alpha(t) \Phi'(t)}{\Phi^{k+1}(t)} dt,\quad \quad \alpha=1,2,\ldots,d.$$ Define
\begin{equation}\label{define}
\gamma_{k,n}^{(\alpha)}:=\frac{1}{2\pi i}\int_{\Gamma_{\rho_1}} \frac{Q_{n,\textup{\textbf{m}}}(t) F_\alpha(t) \Phi'(t)}{\Phi^{k+1}(t)} dt, \quad \quad \alpha=1,2,\ldots,d.
\end{equation}
By our choices of $\rho_1$ and $\rho_2,$ for each $k\geq 0$ and for each $\alpha=1,2,\ldots,d,$ $Q_{n,\textup{\textbf{m}}} F_\alpha \Phi'/\Phi^{k+1}$ is meromorphic in $\overline{D_{\rho_1}}\setminus D_{\rho_2}=\{z\in \mathbb{C}: \rho_2 \leq |\Phi(z)|\leq \rho_1\}$ and has poles at $\lambda_1,\lambda_2,\ldots,\lambda_q$ with multiplicities at most $\tau_1,\tau_2,\ldots,\tau_q.$
Applying Cauchy's residue theorem, we obtain
\begin{equation}\label{residue}
\gamma_{k,n}^{(\alpha)}-a_{k,n}^{(\alpha)}=\sum_{j=1}^{q} \textup{res}(Q_{n,\textup{\textbf{m}}} F_\alpha \Phi'/\Phi^{k+1}, \lambda_j), \quad \quad \alpha=1,2,\ldots,d, \quad  \quad k\geq 0.
\end{equation}
The limit formula for the residue gives 
$$\textup{res}(Q_{n,\textup{\textbf{m}}} F_\alpha \Phi'/\Phi^{k+1}, \lambda_j)=\frac{1}{(\tau_j-1)!} \lim_{z \rightarrow \lambda_j} \left(\frac{(z-\lambda_j)^{\tau_j} Q_{n,\textup{\textbf{m}}}(z) F_\alpha(z) \Phi'(z)}{\Phi^{k+1}(z)}  \right)^{(\tau_j-1)}.$$ Leibniz's formula allows us to write
$$\left(\frac{(z-\lambda_j)^{\tau_j} Q_{n,\textup{\textbf{m}}}(z) F_\alpha(z) \Phi'(z)}{\Phi^{k+1}(z)}  \right)^{(\tau_j-1)}$$
$$=\sum_{t=0}^{\tau_j-1}{\tau_j-1 \choose t} \left(\frac{Q_{n,\textup{\textbf{m}}}(z) \Phi'(z)}{\Phi^{n+1}(z)}  \right)^{(\tau_j-1-t)} \left( (z-\lambda_j)^{\tau_j}F_\alpha(z) \Phi^{n-k}(z) \right)^{(t)}.$$ For $j=1,2,\ldots,q,$ and $t=0,1,\ldots,\tau_j-1,$ set 
$$\beta_n(j,t):=\frac{1}{(\tau_j-1)!} {\tau_j-1 \choose t} \lim_{z \rightarrow \lambda_j}  \left(\frac{Q_{n,\textup{\textbf{m}}}(z) \Phi'(z)}{\Phi^{n+1}(z)}  \right)^{(\tau_j-1-t)}$$
(notice that the $\beta_n(j,t)$ do not depend on $k$ and $\alpha$). So, we rewrite \eqref{residue} as 
\begin{equation}\label{main}
\gamma_{k,n}^{(\alpha)}-a_{k,n}^{(\alpha)}=\sum_{j=1}^{q} \sum_{t=0}^{\tau_j-1} \beta_n(j,t)  \left( (z-\lambda_j)^{\tau_j}F_\alpha(z) \Phi^{n-k}(z) \right)^{(t)}_{z=\lambda_j}, \quad  \alpha=1,2,\ldots,d,   \quad k\geq 0.
\end{equation} By the definition of simultaneous Pad\'{e}-Faber approximants, $$a_{k,n}^{(\alpha)}=0,\quad \quad \alpha=1,2,\ldots,d,\quad\quad  k=n-m_\alpha+1,n-m_\alpha+2,\ldots,n,$$ which implies 
\begin{equation}\label{realmain}
\gamma_{k,n}^{(\alpha)}=\sum_{j=1}^{q} \sum_{t=0}^{\tau_j-1} \beta_n(j,t)  \left( (z-\lambda_j)^{\tau_j}F_\alpha(z) \Phi^{n-k}(z) \right)^{(t)}_{z=\lambda_j}
\end{equation}
for all $\alpha=1,2,\ldots,d$ and $k=n-m_\alpha+1,n-m_\alpha+2,\ldots,n.$ 
We view \eqref{realmain} as a system of $|\textup{\textbf{m}}|$ equations on the $|\textup{\textbf{m}}|$ unknowns $\beta_n(j,t)$ and the determinant corresponding this system is 
\begin{eqnarray}\label{matrixL}
\Delta:=\begin{vmatrix}\notag
 \left[ (z-\lambda_j)^{\tau_j}F_\alpha(z) \Phi^{m_\alpha-1}(z)  \right]_{z=\lambda_j} &     \cdots &  \left[ (z-\lambda_j)^{\tau_j}F_\alpha(z) \Phi^{m_\alpha-1}(z)  \right]^{(\tau_j-1)}_{z=\lambda_j} \\
  \left[ (z-\lambda_j)^{\tau_j}F_\alpha(z) \Phi^{m_\alpha-2}(z)  \right]_{z=\lambda_j}  &  \cdots & \left[ (z-\lambda_j)^{\tau_j}F_\alpha(z) \Phi^{m_\alpha-2}(z)  \right]^{(\tau_j-1)}_{z=\lambda_j}  \\
   \vdots   & \vdots     &  \vdots  \\
  \left[ (z-\lambda_j)^{\tau_j}F_\alpha(z)   \right]_{z=\lambda_j} &  \cdots &    \left[ (z-\lambda_j)^{\tau_j}F_\alpha(z)   \right]^{(\tau_j-1)}_{z=\lambda_j} \\
 \end{vmatrix}_{j=1,\ldots,q,\,\,\alpha=1,\ldots,d},
 \end{eqnarray}
where the subindex on the determinant means that the indicated group of columns are successively written for $j=1,2,\ldots, q$ and the rows repeated for $\alpha=1,2,\ldots,d.$

 If $\Delta=0,$ then there exists a linear combination of rows giving the zero vector. This means that there exist polynomials $v_1(z),v_2(z),\ldots,v_d(z),$ such that $\deg v_\alpha\leq m_{\alpha}-1$ and 
$$  \sum_{\alpha=1}^{d}[ (z-\lambda_j)^{\tau_j} v_\alpha(\Phi(z)) F_\alpha(z)]_{z=\lambda_j}^{(l)}=0,\quad \quad j=1,2,\ldots, q,\quad \quad l=0,1,\ldots, \tau_j-1.$$
Equivalently, $\sum_{\alpha=1}^{d}v_\alpha(\Phi(z)) F_\alpha(z)\in \mathcal{H}(D_{\rho_{|\textup{\textbf{m}}|} (\textup{\textbf{F}})}\setminus E).$ This is impossible because $\textup{\textbf{F}}$ is polewise independent with respect to $\textup{\textbf{m}}$ in $D_{\rho_{|\textup{\textbf{m}}|}(\textup{\textbf{F}})}.$ Therefore, $\Delta\not=0$ and $|\Delta| \geq c_5>0.$

To avoid long expressions, let us define: for all $w=1,2,\ldots,d,$ $y=1,2,\ldots,m_w,$ $j=1,2,\ldots,q,$ and $t=0,1,\ldots,\tau_j-1,$
$$g_{w,y}:=(\sum_{r=0}^{w-1} m_r)+y\quad \quad \textup{and}\quad \quad h_{j,t}:=(\sum_{l=0}^{j-1} \tau_l)+t+1,$$
where $m_0=\tau_0=0.$
Applying Cramer's rule to \eqref{realmain}, we have
\begin{equation}\label{betai}
\beta_n(j,t)=\frac{\Delta_n(j,t)}{\Delta}=\frac{1}{\Delta}\sum_{w=1}^{d}\sum_{y=1}^{m_w} \gamma_{n-m_w+y,n}^{(w)}C[g_{w,y}, h_{j,t}],
\end{equation}
where $\Delta_n(j,t)$ is the determinant obtained from $\Delta$ by replacing the $h_{j,t}^{\textup{th}}$ column with the column $$[\gamma_{n-m_w+1,n}^{(w)} \quad \gamma_{n-m_w+2,n}^{(w)} \quad  \ldots \quad \gamma_{n,n}^{(w)}]_{w=1,2,\ldots,d}^{T}$$ and $C[g,h]$ is the determinant of the $(g,h)^{\textup{th}}$ cofactor matrix of $\Delta_n(j,t).$ Substituting $\beta_n(j,t)$ in \eqref{main}  with the expression in \eqref{betai}, we obtain for $\alpha=1,2,\ldots,d$ and $k\geq n+1,$
\begin{equation}\label{yhb}\gamma_{k,n}^{(\alpha)}-a_{k,n}^{(\alpha)}=\frac{1}{\Delta}\sum_{j=1}^{q} \sum_{t=0}^{\tau_j-1} \sum_{w=1}^{d}\sum_{y=1}^{m_w} \gamma_{n-m_w+y,n}^{(w)}C[g_{w,y}, h_{j,t}]  \left( (z-\lambda_j)^{\tau_j}F_\alpha(z) \Phi^{n-k}(z) \right)^{(t)}_{z=\lambda_j}.
\end{equation}

  Define
  $$\mathbb{B}(\lambda,r):=\{z\in \mathbb{C}: |z-\lambda|<r\}.$$
  Let $\varepsilon>0$ be sufficiently small so that $\{z\in \mathbb{C}: |z-\lambda_j|=\varepsilon\}\subset \{z\in \mathbb{C}:|\Phi(z)|>\rho_2\}$ for all $j=1,2,\ldots,q$ and $\overline{\mathbb{B}(\lambda_j,\varepsilon)}\cap \overline{\mathbb{B}(\lambda_\alpha,\varepsilon)}=\emptyset$ for all $\alpha\not=j.$ Using Cauchy's integral formula,
\begin{equation}\label{cauchy}
\left( (z-\lambda_j)^{\tau_j}F_\alpha(z) \Phi^{n-k}(z) \right)^{(l)}_{z=\lambda_j}=\frac{l!}{2\pi i}\int_{|z-\lambda_j|=\varepsilon}\frac{(z-\lambda_j)^{\tau_j}F_\alpha(z) \Phi^{n-k}(z)dz}{(z-\lambda_j)^{l+1}}.
\end{equation}
We can easily check that there exists a constant $c_6$ such that such that for all  $j=1,2,\ldots,q,$  $l=0,1,\ldots,\tau_j-1,$ $\alpha=1,2,\ldots,d,$ and $k\geq n+1,$
\begin{equation}\label{new1}
\left|\left( (z-\lambda_j)^{\tau_j}F_\alpha(z) \Phi^{n-k}(z) \right)^{(l)}_{z=\lambda_j} \right|\leq  \frac{c_6}{\rho_2^{k-n}}, 
\end{equation} 
for sufficiently large $n.$ Similarly, there exists a constant $c_7$ such that for all  $j=1,2,\ldots,q,$ $l=0,1,\ldots,\tau_j-1,$ $\alpha=1,2,\ldots,d,$ and $k=n-m_\alpha+1,n-m_\alpha+2,\ldots,n,$
\begin{equation}\label{new2}
\left|\left( (z-\lambda_j)^{\tau_j}F_\alpha(z) \Phi^{n-k}(z) \right)^{(l)}_{z=\lambda_j} \right|\leq  c_7,
\end{equation} 
for sufficiently large $n.$
From \eqref{new2}, 
\begin{equation}\label{new4} 
|C(g,h)|\leq c_{8},\quad \quad g,h=1,2,\ldots,|\textup{\textbf{m}}|.
\end{equation}
Using \eqref{new1},  \eqref{new4}, and $\Delta\geq c_5>0,$ it follows from \eqref{yhb} that 
\begin{equation}\label{wow}
|a_{k,n}^{(\alpha)}|\leq |\gamma_{k,n}^{(\alpha)}|+\frac{c_{9}}{\rho_2^{k-n}} \sum_{w=1}^{d}\sum_{y=1}^{m_w} |\gamma_{n-m_w+y,n}^{(w)}|,\quad \quad \alpha=1,2,\ldots,d,\quad \quad  k\geq n+1.
\end{equation}
By the definition of $\gamma_{k,n}^{(\alpha)}$ (see \eqref{define}), for all sufficiently large $n,$ $$|\gamma_{k,n}^{(\alpha)}|\leq \frac{c_{10}}{(\rho_1-\delta)^k}, \quad \quad \alpha=1,2,\ldots,d, \quad \quad  k \geq n-|\textup{\textbf{m}}|+1,
$$ where $\delta$ is sufficiently small so that $\delta$ satisfies \eqref{delta} and $\rho_2<\rho_1-\delta.$ This and the equality \eqref{wow} imply
\begin{equation}\label{aapr}
|a_{k,n}^{(\alpha)}| \leq \frac{c_{11}}{\rho_2^{k-n} (\rho_1-\delta)^n}, \quad \quad \alpha=1,2,\ldots,d, \quad \quad k\geq n+1.
\end{equation}
Moreover, we have for all $\nu\geq 0$ and $k\geq n+1,$ 
 \begin{equation}\label{yar}
 |[ Q_{|\textup{\textbf{m}}|}^{\textup{\textbf{F}}} \Phi_k]_{\nu}|=\left|\frac{1}{2\pi i}\int_{\Gamma_{\rho_2-2\delta}} \frac{Q_{|\textup{\textbf{m}}|}^{\textup{\textbf{F}}}(t) \Phi_k(t) \Phi'(t)}{\Phi^{\nu+1}(t)} dt \right|\leq c_{12} \frac{(\rho_2-\delta)^k}{(\rho_2-3\delta)^{\nu}},
 \end{equation}
 where $\delta$ is sufficiently small so that $\rho_2-3\delta>1.$
Combining \eqref{aapr} and \eqref{yar}, we have for all  $\alpha=1,2,\ldots,d,$
$$|b_{\nu,n}^{(\alpha)}|=\left|\sum_{k=n+1}^{\infty} a_{k,n}^{(\alpha)} [ Q_{|\textup{\textbf{m}}|}^{\textup{\textbf{F}}} \Phi_k]_{\nu}\right|\leq \sum_{k=n+1}^{\infty} |a_{k,n}^{(\alpha)}| |[ Q_{|\textup{\textbf{m}}|}^{\textup{\textbf{F}}} \Phi_k]_{\nu}|$$
$$\leq \frac{c_{13}}{(\rho_2-3\delta)^{\nu} } \left(\frac{\rho_2}{\rho_1-\delta}\right)^n \sum_{k=n+1}^{\infty} \left(\frac{\rho_2-\delta}{\rho_2}\right)^{k} \leq \frac{c_{14 }}{(\rho_2-3\delta)^{\nu} } \left(\frac{\rho_2}{\rho_1-\delta}\right)^n \left(\frac{\rho_2-\delta}{\rho_2}\right)^{n}$$
$$=\frac{c_{14}}{(\rho_2-3\delta)^{\nu} } \left(\frac{\rho_2-\delta}{\rho_1-\delta}\right)^n$$

Now, we show \eqref{hardpart}.
 Recall that our choices of $\rho_1$ and $\delta$ (see \eqref{rho} and \eqref{delta}, respectively) give   $ \|\Phi\|_K+\delta<\rho_1-\delta.$  Moreover, 
$\|\Phi_{\nu}\|_{K}\leq c_2(\|\Phi\|_K+\delta)^{\nu},$ for all  $\nu\geq 0.$
Therefore, for each $\alpha=1,2,\ldots,d,$
$$\|\sum_{\nu=0}^{n+|\textup{\textbf{m}}|-m_\alpha} b_{\nu,n}^{(\alpha)} \Phi_\nu\|_K\leq \sum_{\nu=0}^{n+|\textup{\textbf{m}}|-m_\alpha}|b_{\nu,n}^{(\alpha)}|\|\Phi_\nu\|_K\leq c_{15} \left(\frac{\rho_2-\delta}{\rho_1-\delta}\right)^n \sum_{\nu=0}^{n+|\textup{\textbf{m}}|-m_\alpha} \left(\frac{\|\Phi\|_K+\delta}{\rho_2-3\delta}\right)^{\nu}$$
$$\leq c_{15}(n+|\textup{\textbf{m}}|-m_\alpha+1)\left(\frac{\rho_2-\delta}{\rho_1-\delta}\right)^n\left(\frac{\|\Phi\|_K+\delta}{\rho_2-3\delta}\right)^{n+|\textup{\textbf{m}}|-m_\alpha}.$$
Hence, for each $\alpha=1,2,\ldots,d,$
$$\limsup_{n \rightarrow \infty}\|\sum_{\nu=0}^{n+|\textup{\textbf{m}}|-m_\alpha} b_{\nu,n}^{(\alpha)} \Phi_\nu\|_K^{1/n}\leq \left(\frac{\|\Phi\|_K+\delta}{\rho_1-\delta}\right)  \left(\frac{\rho_2-\delta}{\rho_2-3\delta} \right).$$
Letting $\delta\rightarrow 0$ and $\rho_1\rightarrow \rho_{|\textup{\textbf{m}}|}(\textup{\textbf{F}}),$ for each $\alpha=1,2,\ldots,d,$ 
\begin{equation}\label{finishthedifficultone}
\limsup_{n \rightarrow \infty}\|\sum_{\nu=0}^{n+|\textup{\textbf{m}}|-m_\alpha} b_{\nu,n}^{(\alpha)} \Phi_\nu\|_K^{1/n}\leq \frac{\|\Phi\|_{K}}{\rho_{|\textup{\textbf{m}}|}(\textup{\textbf{F}})}.
\end{equation}
Combining \eqref{simplepart} and \eqref{finishthedifficultone}, we have \eqref{goal}. Therefore, from \eqref{use}, we obtain
\begin{equation}\label{39}
\limsup_{n \rightarrow \infty} \| Q_{|\textup{\textbf{m}}|}^{\textup{\textbf{F}}} Q_{n,\textup{\textbf{m}}}  F_\alpha-Q_{|\textup{\textbf{m}}|}^{\textup{\textbf{F}}}  P_{n,\textup{\textbf{m}},\alpha}\|_K^{1/n}\leq \frac{\|\Phi\|_K}{\rho_{|\textup{\textbf{m}}|}(\textup{\textbf{F}})},\quad \quad \alpha=1,2,\ldots,d,
\end{equation} 
where $K$ is any compact set such that $E\subset K \subset D_{\rho_{|\textup{\textbf{m}}|}(\textup{\textbf{F}})}.$ To show that \eqref{39} is true for any compact subset $K$ of $D_{\rho_{|\textup{\textbf{m}}|}(\textup{\textbf{F}})},$ we let $K$ be any compact subset of $D_{\rho_{|\textup{\textbf{m}}|}(\textup{\textbf{F}})}.$ If $K\subset E,$ then clearly $\|\Phi\|_K$ on the right of \eqref{39} can be replaced by $1.$ If $K\cap (D_{\rho_{|\textup{\textbf{m}}|}(\textup{\textbf{F}})}\setminus E )\not=\emptyset,$ then for any $ \alpha=1,2,\ldots,d,$
$$\limsup_{n \rightarrow \infty} \| Q_{|\textup{\textbf{m}}|}^{\textup{\textbf{F}}} Q_{n,\textup{\textbf{m}}}  F_\alpha-Q_{|\textup{\textbf{m}}|}^{\textup{\textbf{F}}}  P_{n,\textup{\textbf{m}},\alpha}\|_K^{1/n}$$
$$
\leq \limsup_{n \rightarrow \infty} \| Q_{|\textup{\textbf{m}}|}^{\textup{\textbf{F}}} Q_{n,\textup{\textbf{m}}}  F_\alpha-Q_{|\textup{\textbf{m}}|}^{\textup{\textbf{F}}}  P_{n,\textup{\textbf{m}},\alpha}\|_{K\cup E}^{1/n}\leq \frac{\|\Phi\|_{K \cup E}}{\rho_{|\textup{\textbf{m}}|}(\textup{\textbf{F}})}=\frac{\|\Phi\|_{K }}{\rho_{|\textup{\textbf{m}}|}(\textup{\textbf{F}})}.$$ Therefore, \eqref{39} is true for any compact set $K\subset D_{\rho_{|\textup{\textbf{m}}|}(\textup{\textbf{F}})}.$

Let $\varepsilon>0.$ From the second inequality of \eqref{fromnormal}, we obtain
$$\|F_{\alpha}-R_{n,\textup{\textbf{m}},\alpha}\|_{K(\varepsilon)}\leq c_{16}n^{2|\textup{\textbf{m}}|} \|Q_{|\textup{\textbf{m}}|}^{\textup{\textbf{F}}}Q_{n,\textup{\textbf{m}}} F_\alpha-Q_{|\textup{\textbf{m}}|}^{\textup{\textbf{F}}} P_{n,\textup{\textbf{m}},\alpha}\|_K \quad \quad \alpha=1,2,\ldots,d.$$ Using \eqref{39}, we have 
\begin{equation}\label{deleteK}
\limsup_{n \rightarrow \infty} \|  F_\alpha-R_{n,\textup{\textbf{m}},\alpha}\|_{K(\varepsilon)}^{1/n}\leq \frac{\|\Phi\|_K}{\rho_{|\textup{\textbf{m}}|}(\textup{\textbf{F}})},\quad \quad \alpha=1,2,\ldots,d,
\end{equation}
for any compact subset $K$ of $D_{\rho_{|\textup{\textbf{m}}|}(\textup{\textbf{F}})}.$ 
 This implies that for each $\alpha=1,2,\ldots,d,$ 
$$\textup{$h$-$\lim_{n \rightarrow \infty}R_{n,\textup{\textbf{m}},\alpha}=F_\alpha$}$$ 
in $D_{\rho_{|\textup{\textbf{m}}|}(\textup{\textbf{F}})}.$ By Lemma \ref{goncharlemma}, each pole of $\textup{\textbf{F}}$ attracts zeros of $Q_{n,\textup{\textbf{m}}}$ according to its multiplicity. Since $\deg Q_{n,\textup{\textbf{m}}}\leq |\textup{\textbf{m}}|,$ $\deg Q_{n,\textup{\textbf{m}}}=|\textup{\textbf{m}}|$ for sufficiently large $n.$ For such $n$'s, $R_{n,\textup{\textbf{m}}}$ is unique. In fact, if this was not the case we could find an infinite subsequence of indices for which Definition \ref{simuf} has solutions with $\deg Q_{n,\textup{\textbf{m}}}<|\textup{\textbf{m}}|,$ which contradicts what was proved. In what follows, we consider only such $n$'s. Moreover, this means that for sufficiently large $n,$ $$Q_{n,\textup{\textbf{m}}}(z)=\prod_{k=1}^{|\textup{\textbf{m}}|}\left(1-\frac{z}{\lambda_{n,k}} \right),$$
and $$\lim_{n \rightarrow \infty} Q_{n,\textup{\textbf{m}}}(z)=Q_{|\textup{\textbf{m}}|}^{\textup{\textbf{F}}}(z).$$ Because the set of the limit points of zeros of $Q_{n,\textup{\textbf{m}}}$ is $\mathcal{P}_{|\textup{\textbf{m}}|}(\textup{\textbf{F}}),$ the inequality \eqref{deleteK} implies \eqref{2.4}.

Finally, we prove \eqref{2.5}. We first need to show that for $j=1,2,\ldots,q,$
\begin{equation}\label{3.31}
\limsup_{n \rightarrow \infty} |(Q_{n,\textup{\textbf{m}}})^{(k)}(\lambda_j)|^{1/n}\leq \frac{|\Phi(\lambda_j)|}{\rho_{|\textup{\textbf{m}}|}(\textup{\textbf{F}})}, \quad \quad k=0,1,\ldots,\tau_j-1
\end{equation}  
by induction on $k.$
Let $\varepsilon>0$ be sufficiently small so that $\overline{\mathbb{B}(\lambda_j,\varepsilon)}\subset D_{\rho_{|\textup{\textbf{m}}|} (\textup{\textbf{F}})}$ for all $j=1,2,\ldots,q$ and the disks $\overline{\mathbb{B}(\lambda_j,\varepsilon)},$ $j=1,2,\ldots,q,$ are pairwise disjoint. Let $j\in\{1,2,\ldots,q\}.$ There exists $\alpha:=\alpha(j)\in \{1,2,\ldots,d\}$ such that $\lambda_j$ is a pole of $F_{\alpha}$ of order $\tau_j.$  As a consequence of \eqref{39}, we have 
\begin{equation}\label{aaa}
\limsup_{n \rightarrow \infty} \|(z-\lambda_j)^{\tau_j} F_\alpha Q_{n,\textup{\textbf{m}}}-(z-\lambda_j)^{\tau_j} P_{n,\textup{\textbf{m}},\alpha}\|_{\overline{\mathbb{B}(\lambda_j,\varepsilon)}}^{1/n}\leq \frac{\|\Phi\|_{\overline{\mathbb{B}(\lambda_j,\varepsilon)}}}{\rho_{|\textup{\textbf{m}}|}(\textup{\textbf{F}})},
\end{equation} so by Cauchy's integral formula for the derivative, we obtain
\begin{equation}\label{3.32}
\limsup_{n \rightarrow \infty} \|[(z-\lambda_j)^{\tau_j} F_\alpha Q_{n,\textup{\textbf{m}}}-(z-\lambda_j)^{\tau_j} P_{n,\textup{\textbf{m}},\alpha}]^{(k)}\|_{\overline{\mathbb{B}(\lambda_j,\varepsilon)}}^{1/n}\leq \frac{\|\Phi\|_{\overline{\mathbb{B}(\lambda_j,\varepsilon)}}}{\rho_{|\textup{\textbf{m}}|}(\textup{\textbf{F}})},
\end{equation}
for all $k\geq 0.$ Letting $\varepsilon \rightarrow 0^{+},$ the inequality \eqref{aaa} implies that 
$$\limsup_{n \rightarrow \infty} |L_{j} Q_{n,\textup{\textbf{m}}} (\lambda_j)|^{1/n }\leq \frac{|\Phi(\lambda_j)|}{\rho_{|\textup{\textbf{m}}|}(\textup{\textbf{F}})},$$
where $L_{j}:=\lim_{z\rightarrow \lambda_j}(z-\lambda_j)^{\tau_j}F_\alpha(z)\not=0$ (because $F_\alpha$ has a pole of order $\tau_j$ at $\lambda_j$).  Therefore, 
$$\limsup_{n \rightarrow \infty} |Q_{n,\textup{\textbf{m}}}(\lambda_j)|^{1/n}\leq \frac{|\Phi(\lambda_j)|}{\rho_{|\textup{\textbf{m}}|}(\textup{\textbf{F}})},$$
which is the base case. Now, let $r\leq \tau_j-1$ and assume that 
\begin{equation}\label{3.33}
\limsup_{n \rightarrow \infty} |(Q_{n,\textup{\textbf{m}}})^{(k)}(\lambda_j)|\leq \frac{|\Phi(\lambda_j)|}{\rho_{|\textup{\textbf{m}}|}(\textup{\textbf{F}})}, \quad \quad k=0,1,\ldots,r-1.
\end{equation}
Let us show that the above inequality also holds for $k=r.$ Using \eqref{3.32}, since $r<\tau_j,$ we obtain 
\begin{equation}\label{3.34}
\limsup_{n \rightarrow \infty} \left| [(z-\lambda_j)^{\tau_j} F_\alpha Q_{n,\textup{\textbf{m}}}]^{(r)}(\lambda_j) \right|^{1/n} \leq \frac{|\Phi(\lambda_j)|}{\rho_{|\textup{\textbf{m}}|}(\textup{\textbf{F}})}.
\end{equation}
By the Leibniz formula, we have 
$$\left[(z-\lambda_j)^{\tau_j} F_\alpha Q_{n,\textup{\textbf{m}}}\right]^{(r)}(\lambda_j)=\sum_{l=0}^{r} {r \choose l} \left[(z-\lambda_j)^{\tau_j} F_\alpha \right]^{(l)}(\lambda_j) (Q_{n,\textup{\textbf{m}}})^{(r-l)}(\lambda_j).$$
Therefore, by \eqref{3.33}, \eqref{3.34}, and the fact that $L_{j}\not=0,$ it follows that 
$$\lim_{n \rightarrow \infty} \left|(Q_{n,\textup{\textbf{m}}})^{(r)}(\lambda_j) \right|^{1/n}\leq \frac{|\Phi(\lambda_j)|}{\rho_{|\textup{\textbf{m}}|}(\textup{\textbf{F}})},$$
which completes the induction and the proof of \eqref{3.31}.

Using Hermite interpolation, it is easy to construct a basis $\{e_{j,t}\}_{j=1,2,\ldots,q,\, t=0,1,\ldots,\tau_j-1}$ in the space of polynomials of degree at most $|\textup{\textbf{m}}|-1$ satisfying
$$e_{j,t}^{(k)}(\lambda_\ell)=\delta_{\ell,j}\delta_{k,t},\quad \quad 1\leq \ell \leq q, \quad \quad 0\leq k \leq \tau_\ell-1.$$ Then,
$$Q_{n,\textup{\textbf{m}}}(z)=\sum_{j=1}^{q} \sum_{t=0}^{\tau_j-1} (Q_{n,\textup{\textbf{m}}})^{(t)}(\lambda_j)e_{j,t}(z)+C_n Q_{|\textup{\textbf{m}}|}^{\textup{\textbf{F}}}(z),$$
where $C_n=\prod_{j=1}^{q} \lambda_j^{\tau_j}/\prod_{k=1}^{|\textup{\textbf{m}}|}\lambda_{n,k}.$ Using \eqref{3.31}, we have 
$$\limsup_{n \rightarrow \infty} \|Q_{n,\textup{\textbf{m}}}-C_{n} Q_{|\textup{\textbf{m}}|}^{{\textup{\textbf{F}}}}\|^{1/n}\leq \frac{\max_{\lambda\in \mathcal{P}_{|\textup{\textbf{m}}|}(\textup{\textbf{F}})}|\Phi(\lambda)|}{\rho_{|\textup{\textbf{m}}|}(\textup{\textbf{F}})}.$$ Evaluating at zero, we obtain 
$$\limsup_{n \rightarrow \infty}|1-C_n|^{1/n}\leq \frac{\max_{\lambda\in \mathcal{P}_{|\textup{\textbf{m}}|}(\textup{\textbf{F}})}|\Phi(\lambda)|}{\rho_{|\textup{\textbf{m}}|}(\textup{\textbf{F}})}.$$ This implies \eqref{2.5} which completes the proof.

\end{proof}

\section{Acknowledgements}

 \quad\quad   I wish to express my gratitude toward the anonymous referee and the editor for
helpful comments and suggestions leading to improvements of this work. I also
want to thank Prof. Guillermo L\'opez Lagomasino and  Assoc. Prof. Chontita Rattanakul for insight on the topic of this paper  and their suggestions.


\begin{thebibliography}{20}

\bibitem{Suetin} P. K. Suetin, \emph{Series of Faber Polynomials} (Nauka, Moscow, 1984; Gordon and Breach Science Publishers, 1998).

\bibitem{Sutinpade} S. P. Suetin, ``On the convergence of rational approximations to polynomial expansions in domains of meromorphy of a given function," Math. USSR Sb. \textbf{34} (3), 367-381 (1978).

\bibitem{nonlinear} S. P. Suetin, ``On Montessus de Ballore's theorem for rational approximants of orthogonal expansions," Math. USSR Sb. \textbf{42} (3),  399-411 (1982).

\bibitem{Suetin2009}
S. P. Suetin, ``On the existence of nonlinear Pad\'{e}-Chebyshev approximations for analytic functions," Math. Notes \textbf{86} (2), 264-275 (2009).

\bibitem{MorrisSaff} P. R. Graves-Morris and E. B. Saff, ``A de Montessus theorem for vector-valued rational interpolants," Lecture Notes in Math. 1105, Springer, Berlin, 1984, 227-242.


\bibitem{SmirnovLebedev} 
V. I. Smirnov and N. A. Lebedev, \emph{The constructive theory of functions of a complex variable} (Nauka, Moscow, 1964; English transl.,  M.I.T. Press, Cambridge, Mass., 1968).

\bibitem{Gonchar1} A. A. Gonchar, ``On the convergence of generalized Pad\'{e} approximants of meromorphic functions," Math. USSR Sb. \textbf{140} (4), 564-577 (1975). 

\bibitem{Aag81} A. A. Gonchar, ``Poles of rows of the Pad\'{e} table and meromorphic continuation of functions," Sb. Math. \textbf{43} (4) 527-546 (1981).




%\bibitem{Sobczyk} G. Sobczyk, ``Generalized Vandermonde determinants and applications, Aportaciones Matematicas", Serie Comunicaciones \textbf{30} 203-213 (2002).








\end{thebibliography}
\end{document}